\documentclass[a4paper,10pt]{article}
\usepackage[T1,T2A]{fontenc}
\usepackage[cp1251]{inputenc}
\usepackage[english]{babel}
\usepackage{mathtext}
\usepackage{amsfonts}
\usepackage{color}
\usepackage{mathrsfs}
\usepackage{mathtools}
\usepackage{amsmath}
\usepackage{amssymb}
\usepackage{esint}
\usepackage{enumerate}

\usepackage{amsfonts}
\usepackage{amsthm}

\AtBeginDocument{%
   \def\MR#1{}
}

\DeclareMathOperator{\BV}{BV}

\DeclareMathOperator{\Var}{V}
\DeclareMathOperator{\LVar}{LV}

\newcommand{\R}{\mathbb{R}}
\newcommand{\E}{\mathbb{E}}

\newcommand{\Set}[2]{\Big\{{#1}\,\Big|\;{#2}\Big\}}

\newcommand{\eq}[1]{\begin{equation}{#1}\end{equation}}

\DeclareMathOperator{\LCJ}{LCJ}

\newcommand{\var}{\Var}

\DeclareMathOperator{\Metr}{\mathcal{X}}

\newcommand{\Pair}{\mathcal{M}}

\newcommand{\Lip}{\operatorname{Lip}}

\renewcommand{\leq}{\leqslant}
\renewcommand{\geq}{\geqslant}

\newtheorem{theorem}{Theorem}[section]
\newtheorem{Cor}[theorem]{Corollary}
\newtheorem{Th}[theorem]{Theorem}

\title{A random Lipschitz function detecting jumps}
\author{Dmitriy Stolyarov}
\date{\today}

\begin{document}
\allowdisplaybreaks
\maketitle

\begin{abstract}
Recently, A. Tyulenev and the author studied the class of metric spaces~$\Metr$ such that for every mapping~$\gamma \colon [0,1]\to \Metr$, there exists a~$1$-Lipschitz function~$F\colon \Metr\to \R$ that catches the total variation of~$\gamma$, i.e. such that~$\var_{F\circ\gamma} \gtrsim \var_\gamma$. In this short note, we show that a metric space~$\Metr$ enjoys this property if and only if there exists a random~$1$-Lipschitz function~$f$ on~$\Metr$ such that~$\E |f(x) - f(y)| \gtrsim \rho(x,y)$ for every~$x$ and~$y$ in~$\Metr$.
\end{abstract}

\section{The~$\LCJ$ property}\label{S1}
Let~$\Metr=(\Metr,\rho)$ be a metric space, let~$\gamma \colon [0,1] \to \Metr$ be a mapping. The most natural way to define its total variation~$\Var_\gamma$ is to take the supremum of the quantities
\eq{\label{DefinitionOfVariation}
\sum\limits_{i=0}^{N-1} \rho\big(\gamma(t_{i}),\gamma(t_{i+1})\big), 
}
where~$0 \leq t_0 \leq t_1 \ldots \leq t_N \leq 1$ is an arbitrary collection of points. Denote the set of all mappings~$\gamma \colon [0,1] \to \Metr$ with finite total variation by~$\BV([0,1],\Metr)$. Another approach comes from the work of L.~Ambrosio who suggested a definition of~$\Metr$-valued mappings of bounded variation in~\cite{Ambrosio1990}. It rests on the principle that a mapping~$u\colon \Omega \to \R^{d_2}$,~$\Omega \subset \R^{d_1}$ is a domain, is of bounded variation if and only if its composition with any coordinate projection has bounded variation. Ambrosio's definition reads as follows: A mapping~$u\colon \Omega \to \Metr$ has bounded total variation if there exists a finite Borel measure~$\sigma$ on~$\Omega$ such that for any~$1$-Lipschitz function~$f\colon \Metr \to \R$,
\eq{
\big|D (f\circ u)\big|(A) \leq \sigma (A),\qquad A \ \text{is a Borel subset of } \Omega.
}
The symbol~$D$ denotes the distributional gradient. In particular,~$f\circ u$ should be a function of bounded variation. One may show that for one-dimensional mappings~$\gamma \colon [0,1]\to \Metr$, Ambrosio's definition coincides with the `classical' one. 

The measure~$\sigma$ is a mysterious object. Given a mapping~$u$, it is unclear how to construct~$\sigma$. In~\cite{OT2025},  R.~Oleinik and A.~Tyulenev raised the question whether one can get rid of~$\sigma$: What happens if we change Ambrosio's definition and simply require that~$f\circ \gamma$ has bounded variation for any~$1$-Lipschitz~$f\colon \Metr \to \R$? They showed that in the case of \emph{continuous} mappings~$\gamma$, the new definition gives the same class of curves as the classical one (and therefore, the same as Ambrosio's definition): A continuous mapping~$\gamma \colon [0,1] \to \Metr$ has bounded variation if and only if~$f\circ \gamma$ has bounded variation for every Lipschitz mapping~$f\colon \Metr \to \R$. What is more, for any continuous~$\gamma$ and any metric space~$\Metr$ there exists a~$1$-Lipschitz function~$f\colon \Metr \to \R$ such that~$\var_\gamma = \var_{f\circ \gamma}$. See~\cite{Bakhtin2025} for an alternative approach to the phenomenon. 

Slightly later, A. Tyulenev and the author in~\cite{StolyarovTyulenev2026} (see the short report~\cite{StolyarovTyulenev2026bis} as well) showed that the continuity assumption is crucial for the coincidence of the two definitions: There exist a mapping~$\gamma \colon [0,1] \to \ell_2$ with infinite total variation and such that~$f\circ \gamma$ is a mapping of bounded variation for every Lipschitz function~$f\colon \ell_2 \to \R$. This fact urged them to introduce the quantity
\eq{\label{LCJFormula}
\LCJ(\Metr) = \inf \Set{\frac{\LVar_\gamma}{\var_\gamma}}{\gamma\in \BV([0,1],\Metr), \ \Var_\gamma > 0},
}
where~$\LVar_\gamma$ is the Lipschitz variation of~$\gamma$:
\eq{\label{eqq.BLV}
\LVar_{\gamma}:=\sup \Set{\Var_{f \circ \gamma}}{f\colon \Metr \to \R \ \text{is $1$-Lipschitz}}.
}
The paper~\cite{StolyarovTyulenev2026} contains several examples of metric spaces for which the~$\LCJ$ constant is positive, as well as examples for which it vanishes. For instance,~$\LCJ(\ell_2) = 0$. However, it is unclear which functions~$f$ are optimal or almost optimal in the supremum defining~$\LVar_\gamma$.  The purpose of this note is to fill this gap. In a sense, the right formula was already present in Section~$5$ of~\cite{StolyarovTyulenev2026}, where the authors showed~$\LCJ(\Metr) > 0$ whenever~$\Metr$ is an ultrametric space. As it often happens in high-dimensional analysis, one should pick a random~$f$. Namely, in Section~$5$ of~\cite{StolyarovTyulenev2026}, the authors construct a random~$1$-Lipschitz function~$f$ on an ultrametric space~$\Metr$ such that
\eq{\label{RandomFunctionBound}
\E |f(x) - f(y)| \gtrsim \rho(x,y),\qquad \text{for any}\ x,y\in \Metr. 
} 
This immediately yields~$\LCJ(\Metr) > 0$: Given any mapping~$\gamma \in \BV([0,1],\Metr)$, find a partition~$0 \leq t_0 \leq t_1 \ldots \leq t_N \leq 1$ of the interval such that the value of the corresponding sum~\eqref{DefinitionOfVariation} almost equals the variation of~$\gamma$, notice that by~\eqref{RandomFunctionBound}
\eq{
\E \sum\limits_{i=0}^{N-1} \Big| f\big(\gamma(t_{i})\big) - f\big(\gamma(t_{i+1})\big)\Big| \gtrsim \Var_\gamma;
}
this implies the existence of the desired~$1$-Lipschitz function~$f$ such that~$\Var_{f\circ \gamma} \gtrsim \Var_\gamma$. 
The reverse implication presented in this note is slightly subtler.

\paragraph{Acknowledgement.} The author acknowledges the use of ChatGPT (OpenAI) during the development of the proof. All mathematical arguments were checked and are the responsibility of the author.


\section{Main theorem}\label{S2}

Denote the Lipschitz constant of a mapping~$f$ by~$\Lip f$. By a~$1$-Lipschitz function we always mean a function~$f$ with~$\Lip f \leq 1$ (the Lipschitz constant of a~$1$-Lipschitz function may be less than one). Let~$\Metr=(\Metr,\rho)$ be a metric space. 
Fix a point~$x_0\in\Metr$ and set
\eq{\label{eq.K-def}
K=\Set{f\colon\Metr\to\R}{f(x_0)=0,\quad \Lip f \leq 1}.
}
We equip~$K$ with the topology of pointwise convergence.  The choice of the
normalization~$f(x_0)=0$ is harmless, since adding a constant does not change
any difference~$|f(x)-f(y)|$. The space~$K$ is compact.  Indeed,
\eq{
|f(x)|=|f(x)-f(x_0)|\leq \rho(x,x_0),\qquad f\in K,
}
and therefore~$K$ is a closed subset of the product
\eq{
\prod_{x\in\Metr}[-\rho(x,x_0),\rho(x,x_0)].
}
Compactness of~$K$ follows from Tychonoff's theorem.  

\begin{Th}
\label{Th.random-lipschitz}
Let~$\Metr=(\Metr,\rho)$ be a separable metric space such that~$\LCJ(\Metr) \geq c$,~$c\in (0,1)$.  There exists a Borel probability measure~$\nu$ on~$K$
such that
\eq{\label{eq.random-catches}
\int_K |f(x)-f(y)|\,d\nu(f)\geq c\rho(x,y),\qquad \text{for any}\quad x,y\in\Metr.
}
\end{Th}
Equivalently, there exists a random $1$-Lipschitz function~$f\in K$ such that
\eq{
\E |f(x)-f(y)|\geq c\rho(x,y),\qquad  \text{for any}\quad x,y\in\Metr.
}
Since~$\Metr$ is separable, the
pointwise topology on~$K$ is metrizable. We restrict ourselves to separable metric spaces, which is sufficient for the applications considered here, since the image of every mapping of finite variation is separable.


\section{Useful tools}\label{S3}
In this small section, we collect the two known statements we will use in the proof of Theorem~\ref{Th.random-lipschitz}.

First, we will use the following measure
form of the lower bound~$\LCJ(\Metr)\geq c$: For every finite positive Borel
measure~$\mu$ on~$\Metr\times\Metr$ such that
\eq{
0<\int_{\Metr\times\Metr}\rho(x,y)\,d\mu(x,y)<\infty,
}
one has
\eq{\label{eq.measure-LCJ}
\sup_{\Lip f \leq 1}
\int_{\Metr\times\Metr}|f(x)-f(y)|\,d\mu(x,y)
\geq
c\int_{\Metr\times\Metr}\rho(x,y)\,d\mu(x,y).
}
For the proof of this formula, see Corollary~$3.3$ in~\cite{StolyarovTyulenev2026}.

Another tool is the classical minimax theorem from~\cite{KyFan1953}. We state a simplification. 
\begin{Th}[Ky Fan's minimax theorem]\label{KyFanTheorem}
Let~$X$ be a compact convex subset of a locally convex topological space. Let~$Y$ be a convex subset of a linear space. Let~$L\colon X\times Y \to \R$
be a function convex with respect to the first variable and
concave with respect to the second one. Assume that for every~$y\in Y$, the map~$x\mapsto L(x,y)$ is continuous on~$X$. Then,
\eq{
\min\limits_{x\in X}\sup\limits_{y\in Y} L(x,y)= \sup\limits_{y\in Y} \min\limits_{x\in X} L(x,y).
}
\end{Th}

\section{Proof of Theorem~\ref{Th.random-lipschitz}}\label{S4}

Let
\eq{
\Pair=\Set{(x,y)\in\Metr\times\Metr}{x\neq y}
}
and define
\eq{\label{eq.Phi-def}
\Phi(x,y,f)=\frac{|f(x)-f(y)|}{\rho(x,y)},\qquad (x,y)\in\Pair,\quad f\in K.
}
For every fixed~$(x,y)\in\Pair$, the function~$f\mapsto\Phi(x,y,f)$ is
continuous on~$K$, and
\eq{
0\leq \Phi(x,y,f)\leq 1.
}

Denote by~$\Delta_{\rm fin}(\Pair)$ the set of finitely supported probability
measures on~$\Pair$, i.e. finite convex combinations of the Dirac deltas.  After the following change of density
\eq{
\alpha =\frac{\rho}{\int_{\Pair}\rho\,d\mu} \cdot \mu,
}
the estimate~\eqref{eq.measure-LCJ}, restricted to finitely supported measures,
becomes
\eq{\label{eq.game-form}
\inf_{\alpha\in\Delta_{\rm fin}(\Pair)}
\sup_{f\in K}
\int_{\Pair}\Phi(x,y,f)\,d\alpha(x,y)
\geq c.
}
Indeed, if~$\alpha\in\Delta_{\rm fin}(\Pair)$ is given, then taking
$\mu=\alpha/\rho$ gives the reverse change of variables.

Let~$\Delta(K)$ be the space of Borel probability measures on~$K$, equipped
with the weak-* topology.  Since~$K$ is compact metrizable,~$\Delta(K)$ is a
compact convex subset of the locally convex space~$C(K)^*$.  Define
\eq{\label{eq.Psi-def}
\Psi(\nu,\alpha)
=
\int_{\Pair}\int_K \Phi(x,y,f)\,d\nu(f)\,d\alpha(x,y),
\qquad
\nu\in\Delta(K),\quad \alpha\in\Delta_{\rm fin}(\Pair).
}
For every~$\alpha\in\Delta_{\rm fin}(\Pair)$, the map
$\nu\mapsto\Psi(\nu,\alpha)$ is continuous and affine on~$\Delta(K)$.
For every~$\nu\in\Delta(K)$, the map
$\alpha\mapsto\Psi(\nu,
\alpha)$ is affine on~$\Delta_{\rm fin}(\Pair)$.

Moreover, for fixed~$\alpha$,
\eq{\label{eq.dirac-max}
\sup_{\nu\in\Delta(K)}\Psi(\nu,\alpha)
=
\sup_{f\in K}\int_{\Pair}\Phi(x,y,f)\,d\alpha(x,y),
}
because a linear functional on~$\Delta(K)$ attains its supremum on
Dirac measures.  Hence~\eqref{eq.game-form} can be rewritten as
\eq{\label{eq.inf-sup-Psi}
\inf_{\alpha\in\Delta_{\rm fin}(\Pair)}
\sup_{\nu\in\Delta(K)}\Psi(\nu,
\alpha)
\geq c.
}

We now apply Theorem~\ref{KyFanTheorem} in the form
\eq{\label{eq.Ky-Fan}
\sup_{\nu\in\Delta(K)}
\inf_{\alpha\in\Delta_{\rm fin}(\Pair)}\Psi(\nu,
\alpha)
=
\inf_{\alpha\in\Delta_{\rm fin}(\Pair)}
\sup_{\nu\in\Delta(K)}\Psi(\nu,
\alpha);
}
i.e. we use~$-\Psi$ in the role of the function~$L$.
Combining~\eqref{eq.inf-sup-Psi} and~\eqref{eq.Ky-Fan}, we obtain
\eq{
\sup_{\nu\in\Delta(K)}
\inf_{\alpha\in\Delta_{\rm fin}(\Pair)}\Psi(\nu,
\alpha)
\geq c.
}
The function
\eq{
\nu\mapsto \inf_{\alpha\in\Delta_{\rm fin}(\Pair)}\Psi(\nu,
\alpha)
}
is upper semicontinuous, being the infimum of continuous functions.  Since
$\Delta(K)$ is compact, the supremum is attained.  Thus, there exists
$\nu_0\in\Delta(K)$ such that
\eq{\label{eq.nu0-finite-measures}
\Psi(\nu_0,
\alpha)
\geq c,
\qquad \text{for all}\quad \alpha\in\Delta_{\rm fin}(\Pair).
}
Taking~$\alpha=\delta_{(x,y)}$ in~\eqref{eq.nu0-finite-measures}, we obtain
\eq{
\int_K\Phi(x,y,f)\,d\nu_0(f)
\geq c,
\qquad \text{for all}\ (x,y)\in\Pair.
}
Using the definition~\eqref{eq.Phi-def}, this is exactly
\eq{
\int_K |f(x)-f(y)|\,d\nu_0(f)
\geq c\rho(x,y),
\qquad x\neq y.
}
For~$x=y$ the same estimate is trivial.  Hence~$\nu:=\nu_0$ satisfies
\eqref{eq.random-catches}.


\section{A corollary}\label{S5}
\begin{Cor}
If~$\Metr$ is separable and~$\LCJ(\Metr) > 0$, then~$\Metr$ admits a bi-Lipschitz embedding into~$L_1$.
\end{Cor}
\begin{proof}
Assume~$\LCJ(\Metr) > 0$. Theorem~\ref{Th.random-lipschitz} provides us with the probability measure~$\nu$ on~$K$. Any point~$x\in \Metr$ defines a function on~$K$ by the rule~$K\ni f \mapsto f(x)\in \R$. Inequality~\eqref{eq.random-catches} says this mapping is a bi-Lipschitz embedding of~$\Metr$ into~$L_1(\nu)$.
\end{proof}
In particular, the Heisenberg group~$\mathbb{H}^1$ with the Carnot--Carath\'eodory metric does not have the~$\LCJ$ property since it does not admit a bi-Lipschitz  embedding into~$L_1$, see~\cite{CheegerKleiner2010}. This is yet another example of a doubling metric space that does not possess the~$\LCJ$-property; the first one may be found among the so-called Laakso spaces, see~\cite{StolyarovTyulenev2026}.


\bibliography{/Users/mac/Documents/Bib/Mybib_26_7}{}
\bibliographystyle{amsplain}

St. Petersburg State University, Department of Mathematics and Computer Science;

d.m.stolyarov at spbu dot ru.

\end{document}